\documentclass[12pt]{article}

\usepackage{amsfonts,graphicx,amsmath,amssymb,amsthm,geometry,mathtools,authblk,mathrsfs,bbm,dsfont,hyperref,nicefrac,enumerate,comment,xcolor,soul,graphicx,caption,float,subcaption,cite}

\newtheorem{lemma}{Lemma}
\newtheorem{remark}[lemma]{Remark}
\newtheorem{setting}[lemma]{Setting}

\newtheorem{theorem}[lemma]{Theorem}

\begin{document}
	
	\title{Maximum principle for discrete time mean-field stochastic optimal control problems}
	
\date{}

	\author{Arzu Ahmadova$^1$, Nazim I. Mahmudov$^2$
		\bigskip
		\\
	\small{$^1$Faculty of Mathematics, University of Duisburg-Essen, 45127, Essen,
		Germany,\\
		e-mail: $arzu.ahmadova@uni-due.de$\\
		$^2$Department of Mathematics, Eastern Mediterranean University, 99628, T.R. North Cyprus,\\
		e-mail: $nazim.mahmudov@emu.edu.tr$}
	
	\smallskip
	}
	
	\maketitle
		\begin{abstract}
		\noindent In this paper, we study the optimal control of a discrete-time stochastic differential equation (SDE) of mean-field type, where the coefficients can depend on both a function of the law and the state of the process. We establish a new version of the maximum principle for discrete-time stochastic optimal control problems. Moreover, the cost functional is also of the mean-field type. This maximum principle differs from the classical principle since we introduce new discrete-time backward (matrix) stochastic equations. Based on the discrete-time backward stochastic equations where the adjoint equations turn out to be discrete backward SDEs with mean field, we obtain necessary first-order and sufficient optimality conditions for the stochastic discrete optimal control problem. To verify, we apply the result to production and consumption choice optimization problem. 	
	\end{abstract}
	\tableofcontents

	\section{Introduction}
A large number of problems, interesting from a theoretical point of view and important from a practical one, has attracted the attention of many mathematicians and engineers. It is not surprising that there is no field in which extremal problems do not arise, and in which it is not essential to the development of these fields that such problems should be solved. The development of the necessary conditions for an extremum was the elaboration of convex programming theory. A central place in this theory is occupied by the Kuhn-Tucker theorem. The embedding of the theory of optimal control in a general theory of necessary conditions was first carried out by Milyutin and Dubovitskii \cite{milyutin}. The great importance of their work lies in the fact that they succeeded in formulating in a refined form necessary conditions for an extremum which can be applied to a wide class of problems. 

The maximum principle for discrete-time systems has become a subject of great interest. We begin this section by summarizing some seminal articles in this field. In \cite{holtzman} it was shown that the convexity requirement is not applicable to many practical systems. Holtzman and Halkin \cite{holtzman-halkin} extend the applicability to much broader classes of practical systems under the condition of directional convexity being weaker than convexity.  Moreover, Gamkrelidze \cite{Gamkrelidze} proved a maximum principle for systems with phase constraints under a number of assumptions. A number of original ideas related to proving the maximum principle can be found in the works of Rozenoer \cite{rozeoner}. 

Jordon and Polak \cite{jordanpolak} have also considered the problem for optimal discrete systems and derived a stationary principle. They applied similar arguments to those used in deriving the Pontryagin maximum principle for continuous-time systems \cite{pontr}. Butkovski \cite{butkowski} first showed that, in contrast to the continuous case, a direct extension of Pontryagin's maximum principle to discrete systems is in general impossible. Of course, such a property of these systems is of theoretical interest to researchers. He clearly demonstrated some errors in the existing works. The intrinsic reason for the errors is that the significance of convexity has been ignored. Generally speaking, the discrete-time maximum principle fails unless a certain convexity precondition is imposed on the control system. However, in this connection, some researchers have established additional conditions, such as convexity of the set of admissible velocities of the system, directional convexity and $z$-directional convexity, etc., and found that under these conditions the maximum principle is valid for discrete control systems. 

Many results have been done on this topic for different kinds of continuous-time stochastic optimal control problems, for example \cite{songliu,mahbash,zhang,peng,gk,hu-ji-xue,yongsiam,wu}, and discrete-time stochastic optimal
control problems, see \cite{rami,vinter,muller,blot,wang,mur,lin,mah1,mah2019,wuzhang} and the references therein). The main difficulty of the stochastic maximum principle for an optimal control problem governed by continuous-time stochastic It\^{o} equations is that the stochastic It\^{o} integral is only of order $\varepsilon$ ("hidden convexity" fails). Therefore, the usual method of first-order needle variation fails. To overcome this difficulty, one has to study both the first and second order terms in the Taylor expansion of the needle variation, and establish a stochastic maximum principle consisting of two backward stochastic differential equations and a maximum condition with an additional quadratic term in the diffusion coefficients, see \cite{mahbash,yong}. It should be noted that Lin and Zhang \cite{lin} used spike variations to show that the necessary condition for discrete-time stochastic optimal problems is associated with the solutions of a pair of discrete-time backward stochastic equations. On this basis, they obtained the maximum principle for the discrete-time stochastic optimal control problem. 

As for the discrete maximum principle for mean-field stochastic optimal control problems framework, there are a few papers dealing with discrete-time mean-field stochastic optimal control.  Unlike the classical stochastic control problem, mean-field terms appear in the system dynamics and cost function, connecting mean-field theory to stochastic control problems. The stochastic mean-field control problem has been an important research topic since the 1950s. The system state is described by a controlled mean-field stochastic differential equation (MF-SDE), which was first proposed in \cite{pontr}, and the first study on MF-SDEs was published in \cite{halkin}. Since then, many researchers have made numerous contributions to the study of MF-SDEs and related topics, see, e.g. \cite{holtzman,gk,vinter,yong,Djehiche,Buckdahn} and the references cited therein. 

Among the many scientific articles on discrete stochastic maximum principle, we will mention only a few with comparison and relation that motivate this work:
\begin{itemize}
	\item Song and Liu considered in \cite{songliu} the optimal control problem for fully coupled forward–backward stochastic difference equations of mean-field type under weak
	convexity assumption. Note that the form of (3.6) as an adjoint equation which was introduced in \cite{songliu} is one kind of backward stochastic difference equation. This adjoint equation is quite different from our adjoint equation \eqref{bse} studied in this paper. One the one hand, they have different forms, on the other hand, the adjoint equation (3.6) is $\mathcal{F}_{t+1}$-measurable;
	\item In \cite{wuzhang}, Wu and Zhang studied recently discrete-time stochastic optimal control problem with convex control domains, for which necessary condition in the form of Pontryagin’s maximum principle and sufficient condition of optimality are derived. They also pointed out that how to overcome of integrability problem of the solution to the adjoint equation which was not taken into consideration in \cite{songliu}.
	\item Recently, Mahmudov \cite{mah2019} derived the first-order and second-order necessary optimality conditions for discrete-time stochastic optimal control problems by virtue of new discrete-time backward stochastic equation and backward stochastic matrix equation under assumption of the set\\ $(f,\sigma_{1},\sigma_{2},\ldots,\sigma_{d},l)(t,\bar{x}(t), U(t))$ being convex. Unlike \cite{mah2019} and \cite{wuzhang} we study mean-field type discrete-time stochastic maximum principle. 
\end{itemize}

Based on the above considerations, the main purpose of this paper is to construct a rigorous mathematical framework for a mean-field type of discrete-time stochastic optimal control problems and to obtain a rigorous maximum principle in an understandable way.
We study the maximum principle for the optimal control of discrete-time systems described by mean-field stochastic difference equations. As far as we know, there are few results on such stochastic control problems. In fact, discrete-time control systems are of great value in practice. For example, digital control can be formulated as a discrete-time control problem in which the sampled data are obtained at discrete times. In a discrete-time system, the Riccati difference equation plays an important role in synthesizing the optimal control. As pointed out in \cite{wuzhang}, the integrability of the solution to the adjoint equation in discrete-time stochastic
optimal control problem is completely different from that in the continuous-time case. However, we also prove that the solution of the adjoint equation has no problem of integrability.

The main perspectives of our work are systematized as below:
\begin{itemize}
	\item First, to study discrete stochastic optimal control problems, we use the finite approximation method applied in \cite{mur}. We extend this method to study the discrete-time stochastic backward equation and introduce the discrete-time stochastic backward matrix equation; 
	\item Second, we prove that the solution to the adjoint equation has no problem of integrability;
	\item Next, a constructive method is that when the necessary optimality condition are also sufficient under certain assumptions;
	\item Finally, as an application, we adapt the practical application based on Theorem \ref{thm:1st order} and consider the discrete-time system with some risk in the investment process.
\end{itemize}
The structure of the paper is as follows. In Section 2, we formulate the main results and give an example to show the applicability of our results. Section \ref{section1} is devoted to stating main results of this paper. In Section \ref{sec:bsde}, we introduce the discrete-time backward stochastic equation and the discrete-time backward stochastic matrix equation and present the solutions in terms of a fundamental stochastic matrix. In Section \ref{proof}, we prove the discrete-time stochastic maximum principle: a first-order necessary condition for optimality. Section \ref{sec:5} is devoted to the sufficient condition for optimality. Section \ref{sec:7} is devoted to the application of production and consumption choice optimization problems. 

\section{Mathematical description}\label{secmath}
In Section \ref{secmath} we present in Setting \ref{setting} the mathematical framework which we use to study the discrete-time stochastic optimal control problems of mean-field type.
\begin{setting}\label{setting}
	Let $\left\| \cdot \right\|$ be a norm, $\langle \cdot, \cdot \rangle$ be an inner product, let $n_1, n_2 \in \mathbb{N}$ and denote the space of ($n_1\times n_2$)-matrices by $\mathbb{R}^{n_1\times n_2}$, and let $\mathbb{R}^{n_1}\coloneqq \mathbb{R}^{n_1\times 1}$, that is, each element of $\mathbb{R}^{n_1}$ is understood as a column vector, let $I$ be the unit matrix with appropriate dimension. For each matrix $A$, $A^{\intercal}$ denotes the transpose of $A$. Moreover, the forward difference operator $\Delta$ is defined for all $h>0$ as $\Delta f(t)= f(t+h)-f(t)$. For a vector $x\in \mathbb{R}^{n}$ denote by $x^{\intercal}$ its transpose. For a symmetric matrix $A$ and vectors $y,y_1,y_2$ of matching dimensions, we denote $A[y]^2\coloneqq y^{\intercal}Ay$, $A[y_1,y_2]\coloneqq y_1^{\intercal}Ay_2$, $\widehat{f}[t]\coloneqq f(t,\widehat{x}(t),\textbf{E}\widehat{x}(t),\widehat
	{u}(t))$.\\\\
	Let $(\Omega, \mathfrak{F}, \mathbb{P})$ be a complete probability space and $N$ be a positive integer. $\mathbb{T}:=\left\{  t_{k}=t_{0}+kh,\ h>0\right\}  _{k=0}^{N}$, let $\left\lbrace w(t_{k}): k=1,\ldots,N+1 \right\rbrace $ be a sequence of $\mathfrak{F}_{k}$-measurable $\mathbb{R}^{d}$-valued random variables, and let $\mathfrak{F}_{k}\subseteq \mathfrak{F}$ be the $\sigma$-field generated by $w(t_{1}),\ldots,w(t_{k})$, i.e., $\mathfrak{F}_{k}=\sigma\left\lbrace w(t_{1}),\ldots,w(t_{k})\right\rbrace $, $k=1,\ldots,N+1$, and $\mathfrak{F}_{0}=\left\lbrace \emptyset, \Omega \right\rbrace $. Let the expectation operator $\textbf{E}$ be denoted by $\textbf{E}x(t)=\int_{\Omega}x(t)\mathbb{P}(\mathrm{d}w)$ for each $w\in \Omega$. For each $t\geq  0$, $\textbf{E}\left\lbrace\cdot \mid \mathfrak{F}_{t} \right\rbrace  $ is the conditional expectation given by $\mathfrak{F}_{t}$. Assume for all $k\in \mathbb{N}$ that $ w_{h}(t_{k})\coloneqq w(t_{k+1})-w(t_{k})$ satisfies the following conditions:
	\begin{enumerate}
		\item[(wi)] For every $w_{h}\left(  t_{k}\right)  =\left(  w_{h}^{1}\left(
		t_{k}\right),...,w_{h}^{d}\left(  t_{k}\right)  \right)  ,\ w_{h}^{1}\left(
		t_{k}\right)  ,...,w_{h}^{d}\left(  t_{k}\right)  $ are independent
		$\mathbb{R}$-valued random variables.
		
		\item[(wii)] $\textbf{E}\left\{  w_{h}\left(  t_{k}\right)  \mid
		\mathfrak{F}_{k}\right\}  =0$, \quad $\textbf{E}\left\{  \left(  w_{h}^{j}\left(
		t_{k}\right)  \right)  ^{2}\mid\mathfrak{F}_{k}\right\}  =h,\\
		\textbf{E}\left(  w_{h}^{j}\left(  t_{k}\right)  \right)  ^{4}<\infty,$ \quad 
		$\textbf{E}\left( w_{h}^{m}\left(  t_{k}\right)  w_{h}^{l}\left(  t_{k}\right)\right) 
		=\left(  t_{k+1}-t_{k}\right)  \delta_{ml}I$.
	\end{enumerate}
	Moreover, let $\widehat{\mathfrak{F}}_{k}=\sigma\left\{  w_{h}\left(
	t_{k+1}\right),\ldots,w_{h}\left(  t_{N}\right)  \right\}$.  Note that $\mathfrak{F}_{k}$ and $\widehat{\mathfrak{F}}_{k}$ are
	independent. Let for all $v \in \mathbb{R}^{r}$ $\Delta f(t,v)\coloneqq f(t, \widehat{x}(t),\textbf{E}\widehat{x}(t),v)-f(t, \widehat{x}(t),\textbf{E}\widehat{x}(t),\widehat
	{u}(t))$,  and let $\mathbb{F}=\left\{  \mathfrak{F}_{k}:k=0,1,...,N\right\}$ be the set.
	A random variable $z=\left\{  z_{k}:k=0,1,...,N\right\}  $ is called
	$\mathbb{F}$-predictable if the random variable
	$z_{k}$ is $\mathfrak{F}_{k}$-measurable for every $k=0,1,...,N,$. Let $L^{2}(\Omega, \mathfrak{F}_{t_k}, \mathbb{R}^{n})$ be the set of all $\mathbb{R}^{n}$-valued $\mathfrak{F}_{t_k}$-measurable random variables $x(t_{k})$ with $\textbf{E}\|x(t_{k})\|^{2}<\infty$.
\end{setting}

	\section{Statement of main results}\label{section1}

In Section \ref{section1}, we establish a class of discrete-time stochastic nonlinear optimal control problems of mean-field type. The system equation is the following nonlinear stochastic difference equation:
\allowdisplaybreaks
\begin{equation}\label{op1}
	\begin{cases}
		x\left(  t+h\right)  =x(t)+hf\left(t,x(t)
		,\mathbf{E}x(t)  ,u(t)\right)+ \\ \sum\limits_{j=1}^{d}	\sigma^{j}\left(  t,x(t) ,\mathbf{E}x(t),u(t)\right)  w_{h}^{j}\left(  t\right),\\
		\textbf{E}x\left(  t+h\right)  =\textbf{E}x\left(  t\right)
		+h\textbf{E}f\left(  t,x\left(  t\right)  ,\mathbf{E}x\left(  t\right)
		,u\left(  t\right)  \right), \\
		x\left(  t_{0}\right)  =\xi\in\mathbb{R}^{n},\ u\left(  t\right)  \in
		\mathcal{U}(t)\subset\mathbb{R}^{r},\quad t\in \mathbb{T}.
	\end{cases}
\end{equation}
Note that the initial value $\xi$ and $\left\lbrace w(t_{k}), k\in 0,1, \ldots N\right\rbrace $ are assumed to be independent of each other. Let $\varphi
:\mathbb{R}^{n}\times \mathbb{R}^{n}\rightarrow \mathbb{R},\ \left(  l,f,\sigma^{j}\right)  :\mathbb{T}\times
\mathbb{R}^{n}\times \mathbb{R}^{n}\times \mathbb{R}^{r}\rightarrow \mathbb{R}\times \mathbb{R}^{n}\times \mathbb{R}^{n}.$ $\left\{
w\left(  t_{k}\right)  :k=1,\dots,N+1\right\}  $ be a sequence of $\mathfrak{F}_{k}
$-measurable $\mathbb{R}^{d}$-valued random variables. 
Then the optimal control problem minimizes the following expected cost functional defined by:
\begin{equation}\label{cost}
	\mathfrak{J}(u)  =\textbf{E}\varphi \left( x\left(  t_{N+1}\right)
	,\mathbf{E}x\left(  t_{N+1}\right)  \right) +\textbf{E}\sum\limits_{t=t_{0}}^{t_{N}}l\left(  t,x\left(  t\right)  ,\mathbf{E}x\left(  t\right)  ,u\left(
	t\right)  \right)  \longrightarrow\min. \\
\end{equation}
The cost functional \eqref{cost} is also of mean-field type since both the running and terminal cost functions $l$ and $\varphi$ depend on the state process through their expected values.\\
Observe from \eqref{op1} and \eqref{cost} that $\mathbf{x}=\left\{  x\left(  t_{k}\right)  \right\}  _{k=0}^{N+1}$ and
$\mathbf{u}=\left\{  u\left(  t_{k}\right)  \right\}  _{k=0}^{N}$ are the state process
and control process, respectively. Let $\left\{  \mathcal{U}(t_{k})
\right\}  _{k=0}^{N}$ be a sequence of nonempty convex subset of $\mathbb{R}^{r}$. We
introduce the following admissible control set%
\[
\mathfrak{U}_{ad}=\left\{  u=\left\{  u\left(  t_{k}\right)  \right\}  _{k=0}^{N}:u\left(  t_{k}\right)  \in L^{2}\left(  \Omega,\mathfrak{F}_{t_{k}},\mathbb{R}^{r}\right)\quad \text{and} \quad u(t_{k})\in \mathcal{U}(t_{k})\right\} .
\]
The pair $\left( \mathbf{x},\mathbf{u}\right)  $ satisfying the constraints \eqref{op1} is
called an admissible pair, and the pair $\left(  \widehat{\mathbf{x}},\widehat
{\mathbf{u}}\right)  $ which is a solution of the problem \eqref{op1}-\eqref{cost}, is called an optimal pair. Our optimal control problem can be stated as follows:

\textbf{Problem ($\mathcal{DOPC}$)}. Minimize \eqref{cost} over $\mathcal{U}(t)$. Any $\widehat{u}(\cdot)\in \mathcal{U}(t)$ satisfying 

\begin{equation*}
	\mathfrak{J}(\widehat{u}(\cdot))=\inf_{v\in \mathcal{U}(t)}\mathfrak{J}(v(\cdot))
\end{equation*}
is called an optimal control. The corresponding $\widehat{x}(\cdot)=\widehat{x}(\cdot,\widehat{u})$ and $(\widehat{x}(\cdot),\widehat{u}(\cdot))$	are called optimal state process and optimal pair, respectively.

Throughout the paper we use the following assumptions.

\begin{description}
	\item[(A1)]  Let $\psi\coloneqq l,f,\sigma^{j},\varphi$. There exists a constant $L>0$ such that
	\begin{align*}
		\left\Vert \psi\left(  t,x_{1},y_{1},u_{1}\right)  -\psi\left(
		t,x_{2},y_{2},u_{2}\right)  \right\Vert  &  \leq L\left(  \left\Vert
		x_{1}-x_{2}\right\Vert +\left\Vert y_{1}-y_{2}\right\Vert +\left\Vert
		u_{1}-u_{2}\right\Vert \right)  ,\ \\
		\left\Vert \psi\left(  t,0,0,0\right)  \right\Vert  &  \leq L,\quad t\in \mathbb{T},\quad x_{1},x_{2},y_{1},y_{2}\in\mathbb{R}^{n},\ u_{1},u_{2}\in\mathbb{R}^{r}.
	\end{align*}
	
	\item[(A2)] Let $\psi\coloneqq l,f,\sigma^{j},\varphi$ be continuously differentiable with respect to $x$. Moreover, there exists a constant $L_{1}>0$ such that
	\begin{align*}
		\left\Vert \psi_{x}\left(  t,x_{1},y_{1},u_{1}\right)  -\psi_{x}\left(
		t,x_{2},y_{2},u_{2}\right)  \right\Vert  &  \leq L_{1}\left(  \left\Vert
		x_{1}-x_{2}\right\Vert +\left\Vert y_{1}-y_{2}\right\Vert +\left\Vert
		u_{1}-u_{2}\right\Vert \right)  ,\ \\
		\left\Vert \psi_{x}\left(  t,0,0,0\right)  \right\Vert  &  \leq L_{1},\quad t\in \mathbb{T},\quad x_{1},x_{2},y_{1},y_{2}\in\mathbb{R}^{n},\ u_{1},u_{2}\in\mathbb{R}^{r};
	\end{align*}

	\item[(A3)] The set $\left\lbrace \mathcal{U}(t): t\in \mathbb{T}\right\rbrace $ is convex;
	\item[(A4)] The set $(f,\sigma_{1},\sigma_{2},\ldots,\sigma_{d},l)(t,\bar{x}(t), \mathcal{U}(t))$  is convex.
\end{description}
Under the assumptions (A1) and (A2), $J$ is well-defined on $\mathcal{U}$.\\
Now we state the first main result of the paper: first-order necessary conditions for the problem \eqref{op1}-\eqref{cost}.

\begin{theorem}\label{thm:1st order}
	Assume Setting \ref{setting}, let $\left(  \widehat{\mathbf{x}},\widehat{\mathbf{u}}\right)  $ be an optimal pair in problem (\ref{op1}), and assume that assumptions (A1)-(A4) hold. Moreover,
	assume that the function $\varphi$ satisfy the Lipschitz condition in a
	neighborhood of the point $\widehat{x}\left(  t_{N+1}\right)  $ and differentiable at that point. Then there exists a solution $\left(p,q^{1},...,q^{d}\right)  :\mathbb{T}\times\Omega\rightarrow\mathbb{R}^{n}\times\mathbb{R}^{n}\times\mathbb{R}^{n\times d}$ of the discrete-time backward stochastic equation
	\begin{align}\label{bse}
		\begin{cases}
			p\left(  t\right)  =\left(  I+h\widehat{f}_{x}^{\intercal}\left[  t\right]  \right)
			\mathbf{E}\left\{  p\left(  t+h\right)  \mid\mathfrak{F}_{t}\right\}
			+\mathbf{E}\left\{  \widehat{f}_{y}^{\intercal}\left[  t\right]  p\left(  t+h\right)
			\right\}\\
			+\sum_{j=1}^{d}\left(  \widehat{\sigma}_{x}^{j}\left[  t\right]  \right)^{\intercal}q^{j}(t) +\sum_{j=1}^{d}\mathbf{E}\left\{  \left(  \widehat{\sigma}_{y}^{j}\left[  t\right]  \right)
			^{\intercal}q^{j}\left(  t\right)  \right\}-\widehat{l}_{x}[t]-\textbf{E}\widehat{l}_{y}[t],\\
			p\left(  t_{N+1}\right)  =-\varphi_{x}\left(  x\left(  t_{N+1}\right)
			,\mathbf{E}x\left(  t_{N+1}\right)  \right)  -\textbf{E}\varphi_{y}\left(  x\left(
			t_{N+1}\right)  ,\mathbf{E}x\left(  t_{N+1}\right)  \right),  \\
			q^{j}\left(  t\right)  =\mathbf{E}\left\{  p\left(  t+h\right)  w_{h}^{j}\left(  t\right)  \mid\mathfrak{F}_{t}\right\},
		\end{cases}
	\end{align}
 Note that for any $t\in \mathbb{T}$ and for any $v\in \mathcal{U}\left(  t\right)$ the following Hamiltonian function $H$ as follows:
	\begin{equation}\label{max1}
		\left\langle H_{u}\left(  t,\widehat{u}\left(  t\right)  \right)
		,v-\widehat{u}\left(  t\right)  \right\rangle \leq0,\ a.s. ,
	\end{equation}
	where%
	\begin{align*}
		H\left(  t,v\right)   & \coloneqq H\left(  t,p\left(  t+h\right)  ,q\left(  t\right)
		,\widehat{x}\left(  t\right)  ,v\right) \\
		&  =\left\langle \mathbf{E}\left\{  p\left(  t+h\right)  \mid\mathfrak{F}%
		_{t}\right\}  ,hf\left(  t,\widehat{x}\left(  t\right)  ,\mathbf{E}\widehat
		{x}\left(  t\right)  ,v\right)  \right\rangle +\sum_{J=1}^{d}\left\langle
		q^{j}\left(  t\right)  ,\sigma^{j}\left(  t,\widehat{x}\left(  t\right)
		,\mathbf{E}\widehat{x}\left(  t\right)  ,v\right)  \right\rangle \\
		&  -l\left( t,\widehat{x}\left(  t\right)  ,\mathbf{E}\widehat
		{x}\left(  t\right)  ,v\right).
	\end{align*}	
\end{theorem}
\begin{remark}
It is well-known that the adjoint equation in the continuous-time case admits a square-integrable solution under the classical assumptions. However, the solution $\left\lbrace (p(t_{k}),q(t_{k}))\right\rbrace, k\in \mathbb{T}$ of \eqref{bse} has a problem of integrability. Nevertheless, we overcome this problem based on the following discussion. Since $w(t_{k})$ are square-integrable, then $x(t_{k}), k\in \mathbb{T}$ are square integrable. \\
By Assumption $(A2)$ and \eqref{bse}, we obtain that
	\begin{align}\label{01}
		\textbf{E}\left | p\left(  t_{N+1}\right)\right |^{2}<\infty.
	\end{align}
The fact that property (wii) and \eqref{01} imply that 
	\begin{align}\label{02}
	\textbf{E}\left |q^{j}\left(  t_{N}\right)\right |^{2}<\infty.
\end{align}
Substituting \eqref{01} and \eqref{02} into \eqref{bse} and assumption $(A2)$ ensure for all positive integer $N$ that 
\begin{align}
	\textbf{E}\left |p(t_{N})\right |^2<\infty.
\end{align}
Therefore, for all $k\in \mathbb{T}$, $\left\lbrace  (p(t_{k}),q(t_{k}))\right\rbrace $ is square integrable, so $\mathbf{E}|p(t_{k})||x(t_{k})|$ and $\mathbf{E}|q(t_{k})||x(t_{k})|$ are well-defined, which implies that the sufficient condition (Theorem \ref{sufH} works). On the other hand,
$\left\lbrace u(t_{k}), k\in \mathbb{T}\right\rbrace $ are square-integrable, then in general we can still get the square-integrability of $\left\lbrace x(t_{k}), k\in \mathbb{T}\right\rbrace $, and $\left\lbrace (p(t_{k}),q(t_{k}))\right\rbrace, k\in \mathbb{T}$ are square-integrable which also means that the proof of the sufficient condition works. In addition, $\mathbf{E}|p(t_{k})||u(t_{k})|$ and $\mathbf{E}|q(t_{k})||u(t_{k})|$ are well-defined for all admissible control $u=\left\lbrace u(t_{k}), k\in \mathbb{T}\right\rbrace $ which implies that the proof of necessary condition (Theorem \ref{thm:1st order}) works. Therefore, all the expectations involving $\left\lbrace (p(t_{k}),q(t_{k}))\right\rbrace, k\in \mathbb{T}$ in this paper are well-defined.

\end{remark}
\section{Backward Stochastic Difference Equations}\label{sec:bsde}

In this section, we first define the discrete-time mean-field type backward stochastic equations.\\
Let $\phi=\left\{  \phi\left(  t_{k}\right)  \right\}  _{k=0}^{N},$
$\ \psi^{j}=\left\{  \psi^{j}\left(  t_{k}\right)  \right\}  _{k=0}^{N},$
$\phi\left(  t\right)  \in L^{2}\left(  \Omega,\mathfrak{F}_{t},\mathbb{R}%
^{n}\right)  ,$ and $\psi^{j}\left(  t\right)  \in L^{2}\left(  \Omega
,\mathfrak{F}_{t},\mathbb{R}^{n}\right)  $, let $A(t), A_{1}(t) \in L^{2}\left(  \Omega,\mathfrak{F}_{t},\mathbb{R}^{n\times n}\right)$, and for all $j=1,...,d$ $B^{j}(t), B_{1}^{j}(t) \in L^{2}\left(  \Omega,\mathfrak{F}_{t},\mathbb{R}^{n\times n}\right)  $ be uniformly bounded
$\mathbb{R}^{n\times n}$-valued random matrices. Assume that $z\left(  t\right)  $
satisfies the following discrete time difference equation
\begin{align}\label{l1}
	\begin{cases}
		z\left(  t+h\right)  =\left(  I+A\left(  t\right)  \right)  z\left(  t\right)
		+A_{1}\left(  t\right)  \mathbf{E}z\left(  t\right)  +\phi\left(  t\right)\\
		+\sum_{j=1}^{d}	\left(  B^{j}\left(  t\right)  z\left(  t\right)  +B_{1}^{j}\left(  t\right)
		\mathbf{E}z\left(  t\right)  +\psi^{j}\left(  t\right)  \right) w_{h}^{j}\left(  t\right) -l(t) ,\\
		z\left(  t_{0}\right)  =\xi\in\mathbb{R}^{n},\ t\in \mathbb{T}.
	\end{cases}
\end{align}
For arbitrary indices $t_{l},t_{k}\in \mathbb{T}\cup\left\{  t_{N+1}\right\}  $, we
introduce the $n\times n$ matrix $\Phi\left(  t_{l},t_{k}\right)  $,
\begin{align}
	\Phi\left(  t_{l},t_{k}\right)   &  =\left\{
	\begin{tabular}
		[c]{lll}
		$0$ & for & $l<k,$\\
		$I$ & for & $l=k,$\\
		$\Theta\left(  t_{l-1}\right)  \Theta\left(  t_{l-2}\right)  ...\Theta\left(
		t_{k}\right)  $ & for & $l>k,$
	\end{tabular}
	\ \ \ \ \ \ \right. \label{fd1}\\
	\Theta\left(  t_{k}\right)   &  =\left(  I+A\left(  t\right)  \right)
	+A_{1}\left(  t\right)  \mathbf{E}\left\{  \cdot\right\}  +\sum_{j=1}^{d}
	\left(  B^{j}\left(  t\right)  +B_{1}^{j}\left(  t\right)  \mathbf{E}\left\{
	\cdot\right\}  \right)  w_{h}^{j}\left(  t\right)  . \label{fd11}
\end{align}
Observe from \eqref{fd1} and \eqref{fd11} that
\begin{equation}
	\Phi\left(  t_{l},t_{0}\right)  =\Phi\left(  t_{l},t_{k}\right)  \Phi\left(
	t_{k},t_{0}\right)  ,\\ l\geq k\geq0. \label{fd2}
\end{equation}
Then the solution of difference equation \eqref{l1} can be written as
\begin{equation}
	z\left(  t\right)  =\Phi\left(  t,t_{0}\right)  \xi+\sum_{\tau=t_{0}}^{t-h}
	\Phi\left(  t,\tau+h\right)  \left(  \phi\left(  \tau\right)  +\sum_{j=1}^{d}
	\psi^{j}\left(  \tau\right)  w_{h}^{j}\left(  \tau\right)  \right)  ,\ t\in \mathbb{T}.
	\label{fd3}
\end{equation}
Let $h_{N+1}\in L^{2}\left(  \Omega,\mathfrak{F}_{t_{N+1}},\mathbb{R}%
^{n}\right)  $ and $l\left(  t\right)  \in L^{2}\left(  \Omega,\mathfrak{F}_{t},\mathbb{R}^{n}\right)$ for all $t\in \mathbb{T}$. Construct a pair of discrete-time
backward stochastic equations $p=\left(  p\left(  t_{0}\right)  ,p\left(
t_{1}\right)  ,\ldots,p\left(  t_{N+1}\right)  \right)  $, $q^{j}=\left(
q^{j}\left(  t_{0}\right)  ,q^{j}\left(  t_{1}\right)  ,\ldots,q^{j}\left(
t_{N}\right)  \right)  $  corresponding to \eqref{l1} as follows:
\begin{align}\label{dbsde}
	\begin{cases}
		p\left(  t\right) =\left(  I+A^{\intercal}\left(  t\right)  \right)
		\mathbf{E}\left\{  p\left(  t+h\right)  \mid\mathfrak{F}_{t}\right\}
		+\mathbf{E}\left\{  A_{1}^{\intercal}\left(  t\right)  p\left(  t+h\right)
		\right\}\\
		+\sum\limits_{j=1}^{d}\left( \left(  \sigma_{x}^{j}\left[  t\right]  \right)  ^{\intercal}q^{j}\left(
		t\right)  +\mathbf{E}\left\{  \left(  \sigma_{y}^{j}\left[  t\right]  \right)
		^{\intercal}q^{j}\left(  t\right)  \right\}\right)-l(t),\\
		p(t_{N+1})=-h_{N+1}, \quad q^{j}(t)=\textbf{E}\left\lbrace p(t+h)w^{j}_{h}(t) \mid \mathfrak{F}_{t}\right\rbrace,\quad  t \in \mathbb{T}.
	\end{cases}
\end{align}

\begin{lemma}[\cite{mah1}]\label{Lem:0}  
	Discrete-time backward stochastic equation
	\eqref{dbsde} has a unique solution $\left(  p,q\right)  $ such that $p\left(
	t\right)  \in L^{2}\left(  \Omega,\mathfrak{F}_{t},\mathbb{R}^{n}\right)
	,q\left(  t\right)  \in L^{2}\left(  \Omega,\mathfrak{F}_{t},\mathbb{R}%
	^{n\times d}\right)  $, and has the following representation%
	\begin{align}
		p\left(  t\right)   &  =-\mathbf{E}\left\{  \Phi^{\intercal}\left(
		t_{N+1},t\right)  h_{N+1}+\sum_{s=t}^{t_{N}}
		\Phi^{\intercal}\left(  s,t\right)  l\left(  s\right)  \mid\mathfrak{F}_{t}\right\}  ,\ \ \nonumber\\
		\ p\left(  t_{N+1}\right)&  =-h_{N+1},\ q^{j}\left(  t\right)
		=\mathbf{E}\left\{  p\left(  t+h\right)  w_{h}^{j}\left(  t\right)
		\mid\mathfrak{F}_{t}\right\}  ,\ \label{pq}
	\end{align}
	where $t\in \mathbb{T},\ j=1,\ldots,d.$
\end{lemma}

\begin{remark}
	Assume that $\Phi^{\intercal}\left(
	t_{k},0\right)  ,$ $t_{k}\in \mathbb{T},$ is invertible we have
	\begin{align*}
		p\left(  t_{k}\right)&=-\mathbf{E}\left\{  \Phi^{\intercal}\left(
		t_{N+1},t\right)  h_{N+1}+\sum_{s=t}^{t_{N}}
		\Phi^{\intercal}\left(  s,t\right)  l\left(  s\right)  \mid\mathfrak{F}_{t_{k}}\right\} \\
		&  =-\left(  \Phi^{\intercal}\left(  t,0\right)  \right)  ^{-1}\mathbf{E}
		\left\{  \Phi^{\intercal}\left(  t_{N+1},0\right)  h_{N+1}+\sum_{s=t}^{t_{N}}
		\Phi^{\intercal}\left(  s,0\right)  l\left(  s\right)  \mid\mathfrak{F}
		_{t_{k}}\right\}.
	\end{align*}
\end{remark}

\section{Proof of Theorem \ref{thm:1st order}}\label{proof}

Assume that $\widehat{u}=\left\{  \widehat{u}\left(  t\right)  \right\}
_{t=t_{0}}^{t_{N}}$ is the optimal control of the problem (\ref{op1}) and
$\widehat{x}=\left\{  \widehat{x}\left(  t\right)  \right\}  _{t=t_{0}%
}^{t_{N+1}}$ is the corresponding optimal trajectory. We fix a time $t_{0}%
\leq\theta\leq t_{N}$ and choose $v\in L^{2}\left(  \Omega,\mathfrak{F}%
_{\theta},\mathbb{R}^{r}\right)  $ such that $u\left(  \theta\right)  +\Delta
v\in \mathcal{U}\left(  \theta\right)  $. For any $\varepsilon>0$, we define the
perturbed admissible control
\[
u^{\varepsilon}\left(  t\right)  =\widehat{u}\left(  t\right)  +\delta
_{t\theta}\varepsilon\Delta v,\ \ t\in \mathbb{T},
\]
where $\delta_{t\theta}=1$ for $t=\theta,$ $\delta_{t\theta}=0$ for
$t\neq\theta$. Convexity of $\mathcal{U}\left(  \theta\right)  $ implies that the
control $u^{\varepsilon}=\left\{  u^{\varepsilon}\left(  t\right)  \right\}
_{t=t_{0}}^{t_{N}}$ is admissible. Let $x^{\varepsilon}$ be a solution of
\eqref{op1} corresponding to the control $u^{\varepsilon}$.

We introduce the following short-hand notations for $f,\ \sigma^{j}$, and $l$:
\begin{align*}
	\widehat{\psi}\left[  t\right]   &  \coloneqq\psi\left(  t,\widehat{x}\left(  t\right)
	,\mathbf{E}\widehat{x}\left(  t\right)  ,\widehat{u}\left(  t\right)  \right)
	,\ \psi^{\varepsilon}\left[  t\right] \coloneqq\psi\left(  t,x^{\varepsilon}\left(
	t\right)  ,\mathbf{E}x^{\varepsilon}\left(  t\right)  ,u^{\varepsilon}\left(
	t\right)  \right);\\
	\widehat{\psi}^{\varepsilon}\left[  t\right]   &  \coloneqq\psi\left(  t,\widehat{x}\left(
	t\right)  ,\mathbf{E}\widehat{x}\left(  t\right)  ,u^{\varepsilon}\left(
	t\right)  \right);\\
	 \widehat{\psi}_{u}\left[  t;\varepsilon\right]  &\coloneqq\int
	_{0}^{1}\psi_{u}\left(  t,\widehat{x}\left(  t\right)  ,\mathbf{E}\widehat
	{x}\left(  t\right)  ,\widehat{u}\left(  t\right)  +\lambda\left(
	u^{\varepsilon}\left(  t\right)  -\widehat{u}\left(  t\right)  \right)
	\right)  d\lambda;\\
	\widehat{z}\left(  t;\varepsilon\right)   &  :=\widehat{z}\left(  t\right)
	+\lambda\left(  z^{\varepsilon}\left(  t\right)  -\widehat{z}\left(  t\right)
	\right)  ,\quad \text{where} \quad z=x,y,u;\\
	\widehat{\psi}_{z}\left[  t;\varepsilon\right]   &  :=\int_{0}^{1}\psi_{z}\left(
	t,\widehat{x}\left(  t\right)  +\lambda\left(  x^{\varepsilon}\left(
	t\right)  -\widehat{x}\left(  t\right)  \right)  ,\mathbf{E}\widehat{x}\left(
	t\right)  +\lambda\left(  \mathbf{E}x^{\varepsilon}\left(  t\right)
	-\mathbf{E}\widehat{x}\left(  t\right)  \right)  ,\widehat{u}\left(  t\right)
	\right)  d\lambda,
\end{align*}
where $z=x$ or $z=y:=\mathbf{E}x$ in the last definition.
\begin{lemma}\label{lem:1}
	Assume that assumption (A1) holds. Then we have for all $\varepsilon>0$ that
	\begin{equation}\label{eq1}
		\max_{0\leq k\leq N}\mathbf{E}\left\Vert x^{\varepsilon}\left(  t_{k}\right)
		-\widehat{x}\left(  t_{k}\right)  \right\Vert ^{2}\leq L\varepsilon
		^{2}\mathbf{E}\left\Vert \Delta v\right\Vert ^{2}.	
	\end{equation}		
\end{lemma}

\begin{proof}
	For $t=t_{0},...,\theta-h,$ it is clear that $x^{\varepsilon}\left(  t+h\right)
	=\widehat{x}\left(  t+h\right)  ,\ \mathbf{E}x^{\varepsilon}\left(  t+h\right)
	=\mathbf{E}\widehat{x}\left(  t+h\right)$. By recursive iteration for $t=\theta,$ we have
	\begin{equation*}
		x^{\varepsilon}\left(  \theta+h\right)  -\widehat{x}\left(  \theta+h\right)
		=\widehat{f}^{\varepsilon}\left[  \theta\right]  -\widehat{f}\left[
		\theta\right]  +
		\sum_{j=1}^{d}\left(  \widehat{\sigma}^{\varepsilon j}\left[  \theta\right]  -\widehat{\sigma}^{j}\left[  \theta\right]  \right)  w_{h}^{j}\left(  \theta\right)  .
	\end{equation*}
	Then observe for all $\varepsilon>0$ that
	\begin{equation*}
		\mathbf{E}\left\Vert x^{\varepsilon}\left(  \theta+h\right)  -\widehat
		{x}\left(  \theta+h\right)  \right\Vert ^{2}\leq \mathbf{E}\left\Vert \widehat
		{f}^{\varepsilon}\left[  \theta\right]  -\widehat{f}\left[  \theta\right]
		\right\Vert ^{2}+\mathbf{E}\sum_{j=1}^{d}\left\Vert \widehat{\sigma}^{\varepsilon j}\left[  \theta\right]
		-\widehat{\sigma}^{j}\left[  \theta\right]  \right\Vert ^{2}.
	\end{equation*}
	By boundedness of $f_{u}$ and $\sigma_{u}^{j},$ and assumption (A1) we have that
	\begin{equation*}
		\mathbf{E}\left\Vert x^{\varepsilon}\left(  \theta+h\right)  -\widehat
		{x}\left(  \theta+h\right)  \right\Vert ^{2}\leq L\varepsilon^{2}\mathbf{E}\left\Vert \Delta v\right\Vert ^{2}.
	\end{equation*}
	For $t=\theta+h$ by boundedness of $f_{x}$ and $\sigma_{x}^{j}$ and assumption (A1), we have that
	\begin{align*}
		&\mathbf{E}\left\Vert x^{\varepsilon}\left(  \theta+2h\right)  -\widehat
		{x}\left(  \theta+2h\right)  \right\Vert ^{2} \\ &\leq\mathbf{E}\left\Vert
		x^{\varepsilon}\left(  \theta+h\right)  -\widehat{x}\left(  \theta+h\right)
		\right\Vert ^{2}\\
		&  +\mathbf{E}\left\Vert f^{\varepsilon}\left(  \theta+h,x^{\varepsilon}\left(
		\theta+h\right)  ,\mathbf{E}x^{\varepsilon}\left(  \theta+h\right)
		,\widehat{u}\left(  \theta+h\right)  \right)  -\widehat{f}\left[
		\theta+h\right]  \right\Vert ^{2}\\
		&  +\mathbf{E}\sum_{j=1}^{d}
		\left\Vert \sigma^{\varepsilon j}\left(  \theta+h,x^{\varepsilon}\left(  \theta+h\right)
		,\mathbf{E}x^{\varepsilon}\left(  \theta+h\right)  ,\widehat{u}\left(
		\theta+h\right)  \right)  -\widehat{\sigma}^{j}\left[  \theta+h\right]
		\right\Vert ^{2}\\
		&  \leq\mathbf{E}\left\Vert x^{\varepsilon}\left(  \theta+h\right)
		-\widehat{x}\left(  \theta+h\right)  \right\Vert ^{2}+2L\mathbf{E}\left\Vert
		x^{\varepsilon}\left(  \theta+h\right)  -\widehat{x}\left(  \theta+h\right)
		\right\Vert ^{2}\\
		&+2L\left\Vert \mathbf{E}x^{\varepsilon}\left(  \theta+h\right)  -\mathbf{E}
		\widehat{x}\left(  \theta+h\right)  \right\Vert ^{2}+2L\varepsilon
		^{2}\mathbf{E}\left\Vert \Delta v\right\Vert ^{2} \leq L\varepsilon
		^{2}\mathbf{E}\left\Vert \Delta v\right\Vert ^{2}.
	\end{align*}
	Similarly, by assumption (A1) we also have that 
	\begin{align*}
		&\left\Vert \mathbf{E}x^{\varepsilon}\left(  \theta+2h\right)  -\mathbf{E}
		\widehat{x}\left(  \theta+2h\right)  \right\Vert ^{2} \\
		&\leq\left\Vert
		\mathbf{E}x^{\varepsilon}\left(  \theta+h\right)  -\mathbf{E}\widehat
		{x}\left(  \theta+h\right)  \right\Vert ^{2}\\
		&  +\left\Vert \mathbf{E}f\left(  \theta+h,x^{\varepsilon}\left(
		\theta+h\right)  ,\mathbf{E}x^{\varepsilon}\left(  \theta+h\right)
		,\widehat{u}\left(  \theta+h\right)  \right)  -\mathbf{E}\widehat{f}\left[
		\theta+h\right]  \right\Vert ^{2}\\
		&\leq\left\Vert \mathbf{E}x^{\varepsilon}\left(  \theta+h\right)
		-\mathbf{E}\widehat{x}\left(  \theta+h\right)  \right\Vert ^{2}+L\mathbf{E}
		\left\Vert x^{\varepsilon}\left(  \theta+h\right)-\widehat{x}\left(
		\theta+h\right)\right\Vert ^{2}\\
		&+L\left\Vert \mathbf{E}x^{\varepsilon}\left(  \theta+h\right)  -\mathbf{E}\widehat{x}\left(  \theta+h\right)  \right\Vert ^{2}\leq L\varepsilon
		^{2}\mathbf{E}\left\Vert \Delta v\right\Vert ^{2}.
	\end{align*}
	Therefore, by recursive iteration for $t=\theta+2h, \ldots, t_{N}$, we obtain the desired result \eqref{eq1}.	The proof of  Lemma \ref{lem:1} is thus complete.	
\end{proof}

\subsection{Duality analysis}
Let $\xi=\left\{  \xi\left(  t\right)  \right\}_{t=t_{0}}^{t_{N}}$ be the
solution to the following difference equation,
\begin{align}\label{xi}
	\begin{cases}
		\xi\left(t+h\right)&  =\xi\left(  t\right)  +\widehat{f}_{x}\left[  t\right]
		\xi\left(  t\right)  +\widehat{f}_{y}\left[  t\right]  \textbf{E}\xi\left(  t\right)
		+\delta_{t\theta}\widehat{f}_{u}\left[  t\right]  \varepsilon\Delta v\\
		&+
		\sum_{j=1}^{d}	\left(  \widehat{\sigma}_{x}^{j}\left[  t\right]  \xi\left(  t\right)  +\widehat{\sigma}_{y}^{j}\left[  t\right]  \mathbf{E}\xi\left(  t\right)  +\delta_{t\theta}\widehat{\sigma}
		_{u}^{j}\left[  t\right]  \varepsilon\Delta v\right)  w_{h}^{j}\left(
		t\right)  ,\\
		\xi\left(  0\right)&=0.
	\end{cases}
\end{align}
By Lemma \ref{lem:1}, observe for all $\varepsilon>0$ that
\begin{equation*}
	\max_{t\in \mathbb{T}}\mathbf{E}\left\Vert \xi\left(  t\right)  \right\Vert ^{2}\leq
	L\varepsilon^{2}\mathbf{E}\left\Vert \Delta v\right\Vert ^{2}.
\end{equation*}

\begin{lemma}\label{lem:2}
	Assume that assumption (A1) holds. Then we have for all $\varepsilon>0$ that
	\[
	\max_{t\in \mathbb{T}}\mathbf{E}\left\Vert x^{\varepsilon}\left(  t\right)
	-\widehat{x}\left(  t\right)  -\xi\left(  t\right)  \right\Vert ^{2}=o\left(
	\varepsilon^{2}\right).
	\]
	
\end{lemma}

\begin{proof}
	For $t=t_{0},...,\theta,$ observe that $ x^{\varepsilon}\left(  t\right)
	-\widehat{x}\left(  t\right)  -\xi\left(  t\right)  =0$ and $\textbf{E}x^{\varepsilon
	}\left(  t\right)  -\mathbf{E}\widehat{x}\left(  t\right)  -\mathbf{E}
	\xi\left(  t\right)  =0$.\\
	By recursive iteration, for $t=\theta$, we have 
	\begin{align*}
		&  x^{\varepsilon}\left(  \theta+h\right)  -\widehat{x}\left(  \theta
		+h\right)  -\xi\left(  \theta+h\right) \\
		&  =\left(  \widehat{f}_{u}\left[  \theta;\varepsilon\right]  -\widehat{f}%
		_{u}\left[  \theta\right]  \right)  \varepsilon\Delta v+\left(  \widehat
		{f}_{v}\left[  \theta;\varepsilon\right]  -\widehat{f}_{v}\left[
		\theta\right]  \right)  \varepsilon\Delta\mathbf{E}v\\
		&  +\sum_{j=1}^{d}
		\left(  \widehat{\sigma}_{u}^{j}\left[  \theta;\varepsilon\right]
		-\widehat{\sigma}_{u}^{j}\left[  \theta\right]  \right)  \varepsilon\Delta
		vw_{h}^{j}\left(  \theta\right)  +\sum_{j=1}^{d}
		\left(  \widehat{\sigma}_{v}^{j}\left[  \theta;\varepsilon\right]
		-\widehat{\sigma}_{v}^{j}\left[  \theta\right]  \right)  \varepsilon
		\Delta\mathbf{E}vw_{h}^{j}\left(  \theta\right).
	\end{align*}
	Then we have
	\begin{gather*}
		\mathbf{E}\left\Vert x^{\varepsilon}\left(  \theta+h\right)  -\widehat
		{x}\left(  \theta+h\right)  -\xi\left(  \theta+h\right)  \right\Vert ^{2}\leq
		L\varepsilon^{2}\mathbf{E}\left\Vert \widehat{f}_{u}\left[  \theta
		;\varepsilon\right]  -\widehat{f}_{u}\left[  \theta\right]  \right\Vert
		^{2}\left\Vert \Delta v\right\Vert ^{2}\\
		+L\varepsilon^{2}\mathbf{E}%
		{\displaystyle\sum_{j=1}^{d}}
		\left\Vert \widehat{\sigma}_{u}^{j}\left[  \theta;\varepsilon\right]
		-\widehat{\sigma}_{u}^{j}\left[  \theta\right]  \right\Vert ^{2}\left\Vert
		\Delta v\right\Vert ^{2}.
	\end{gather*}
	Similarly, we have 
	\begin{align*}
		\left\Vert \mathbf{E}x^{\varepsilon}\left(  \theta+h\right)  -\mathbf{E}%
		\widehat{x}\left(  \theta+h\right)  -\mathbf{E}\xi\left(  \theta+h\right)
		\right\Vert ^{2}\leq L\varepsilon^{2}\left\Vert \mathbf{E}\widehat{f}%
		_{u}\left[  \theta;\varepsilon\right]  -\mathbf{E}\widehat{f}_{u}\left[
		\theta\right]  \right\Vert ^{2}\left\Vert \Delta v\right\Vert ^{2}.
	\end{align*}
	Since $\left\Vert \widehat{f}_{u}\left[  \theta;\varepsilon\right]
	-\widehat{f}_{u}\left[  \theta\right]  \right\Vert ^{2},\left\Vert
	\widehat{\sigma}_{u}^{j}\left[  \theta;\varepsilon\right]  -\widehat{\sigma
	}_{u}^{j}\left[  \theta\right]  \right\Vert ^{2}\rightarrow0$ as
	$\varepsilon\rightarrow0^{+},$ we get%
	\begin{align*}
		\lim_{\varepsilon\rightarrow0}\frac{1}{\varepsilon^{2}}\mathbf{E}\left\Vert
		x^{\varepsilon}\left(  \theta+h\right)  -\widehat{x}\left(  \theta+h\right)
		-\xi\left(  \theta+h\right)  \right\Vert ^{2}  &  =0,\\
		\lim_{\varepsilon\rightarrow0}\frac{1}{\varepsilon^{2}}\left\Vert
		\mathbf{E}x^{\varepsilon}\left(  \theta+h\right)  -\mathbf{E}\widehat
		{x}\left(  \theta+h\right)  -\textbf{E}\xi\left(  \theta+h\right)  \right\Vert ^{2}
		&  =0.
	\end{align*}
	For $t=\theta+2h,...,t_{N},$
	\begin{align*}
		&  x^{\varepsilon}\left(  t+h\right)  -\widehat{x}\left(  t+h\right)
		-\xi\left(  t+h\right) \\
		&  =\widehat{f}_{x}\left[  t;\varepsilon\right]  \left(  x^{\varepsilon
		}\left(  t\right)  -\widehat{x}\left(  t\right)  -\xi\left(  t\right)
		\right)  +\widehat{f}_{y}\left[  \theta;\varepsilon\right]  \left(
		\mathbf{E}x^{\varepsilon}\left(  t\right)  -\mathbf{E}\widehat{x}\left(
		t\right)  -\mathbf{E}\xi\left(  t\right)  \right) \\
		&  +\left(  \widehat{f}_{x}\left[  t;\varepsilon\right]  -\widehat{f}%
		_{x}\left[  t\right]  \right)  \xi\left(  t\right)  +\left(  \widehat{f}%
		_{y}\left[  t;\varepsilon\right]  -\widehat{f}_{y}\left[  t\right]  \right)
		\mathbf{E}\xi\left(  t\right) \\
		&  +\sum_{j=1}^{d}
		\widehat{\sigma}_{x}^{j}\left[  \theta;\varepsilon\right]  \left(
		x^{\varepsilon}\left(  t\right)  -\widehat{x}\left(  t\right)  -\xi\left(
		t\right)  \right)  w_{h}^{j}\left(  t\right)  +\sum_{j=1}^{d}
		\widehat{\sigma}_{y}^{j}\left[  \theta;\varepsilon\right]  \left(
		\mathbf{E}x^{\varepsilon}\left(  t\right)  -\mathbf{E}\widehat{x}\left(
		t\right)  -\mathbf{E}\xi\left(  t-h\right)  \right)  w_{h}^{j}\left(  t\right)
		\\
		&  +\sum_{j=1}^{d}
		\left(  \widehat{\sigma}_{x}^{j}\left[  t;\varepsilon\right]  -\widehat
		{\sigma}_{x}^{j}\left[  t\right]  \xi\left(  t\right)  w_{h}^{j}\left(
		t\right)  \right)  +\sum_{j=1}^{d}
		\left(  \widehat{\sigma}_{y}^{j}\left[  t;\varepsilon\right]  -\widehat
		{\sigma}_{y}^{j}\left[  t\right]  \mathbf{E}\xi\left(  t\right)  w_{h}^{j}\left(  t\right)  \right).
	\end{align*}
	Then we have
	\begin{align*}
		&  \mathbf{E}\left\Vert x^{\varepsilon}\left(  t+h\right)  -\widehat{x}\left(
		t+h\right)  -\xi\left(  t+h\right)  \right\Vert ^{2}\\
		&  \leq L\mathbf{E}\left\Vert x^{\varepsilon}\left(  t\right)  -\widehat
		{x}\left(  t\right)  -\xi\left(  t\right)  \right\Vert ^{2}+L\left\Vert
		\mathbf{E}x^{\varepsilon}\left(  t\right)  -\mathbf{E}\widehat{x}\left(
		t\right)  -\mathbf{E}\xi\left(  t\right)  \right\Vert ^{2}\\
		&  +\mathbf{E}\left(  \left\Vert \widehat{f}_{x}\left[  t;\varepsilon\right]
		-\widehat{f}_{x}\left[  t\right]  \right\Vert ^{2}+\sum_{j=1}^{d}
		\left\Vert \widehat{\sigma}_{x}^{j}\left[  t;\varepsilon\right]
		-\widehat{\sigma}_{x}^{j}\left[  t\right]  \right\Vert ^{2}\right)  \left\Vert
		\xi\left(  t\right)  \right\Vert ^{2}\\
		&  +\mathbf{E}\left(  \left\Vert \widehat{f}_{y}\left[  t;\varepsilon\right]
		-\widehat{f}_{y}\left[  t\right]  \right\Vert ^{2}+\sum_{j=1}^{d}
		\left\Vert \widehat{\sigma}_{y}^{j}\left[  t;\varepsilon\right]
		-\widehat{\sigma}_{y}^{j}\left[  t\right]  \right\Vert \right)  \left\Vert
		\mathbf{E}\xi\left(  t\right)  \right\Vert ^{2}.
	\end{align*}
	By assumption (A1), we obtain as $\varepsilon\rightarrow 0^{+}$ that
	\begin{align*}
		\left\Vert \widehat{f}_{x}\left[  t;\varepsilon\right]  -\widehat{f}%
		_{x}\left[  t\right]  \right\Vert ^{2}+\sum_{j=1}^{d}\left\Vert \widehat{\sigma}_{x}^{j}\left[  t;\varepsilon\right]
		-\widehat{\sigma}_{x}^{j}\left[  t\right]  \right\Vert ^{2}  &  \rightarrow
		0,\\
		\left\Vert \widehat{f}_{y}\left[  t;\varepsilon\right]  -\widehat{f}%
		_{y}\left[  t\right]  \right\Vert ^{2}+%
		\displaystyle\sum_{j=1}^{d}	\left\Vert \widehat{\sigma}_{y}^{j}\left[  t;\varepsilon\right]
		-\widehat{\sigma}_{y}^{j}\left[  t\right]  \right\Vert ^{2}  &  \rightarrow0,
	\end{align*}
	 Therefore, we get that
	\begin{gather*}
		0\leq\lim_{\varepsilon\rightarrow0}\frac{1}{\varepsilon^{2}}\mathbf{E}%
		\left\Vert x^{\varepsilon}\left(  t+h\right)  -\widehat{x}\left(  t+h\right)
		-\xi\left(  t+h\right)  \right\Vert ^{2}\\
		\leq\lim_{\varepsilon\rightarrow0}\frac{1}{\varepsilon^{2}}\mathbf{E}%
		\left\Vert x^{\varepsilon}\left(  t\right)  -\widehat{x}\left(  t\right)
		-\xi\left(  t\right)  \right\Vert ^{2}\\
		+\lim_{\varepsilon\rightarrow0}\frac{1}{\varepsilon^{2}}\mathbf{E}\left\Vert
		x^{\varepsilon}\left(  t-h\right)  -\widehat{x}\left(  t-h\right)  -\xi\left(
		t-h\right)  \right\Vert ^{2}=0.
	\end{gather*}
	Finally, the conclusion is obtained by induction.	
\end{proof}

\begin{lemma}\label{lem:3}
	We have the first-order increment of cost functional \eqref{cost} as follows:
	\begin{align}\label{deltacost}
		&  J\left(  u^{\varepsilon}\right)  -J\left(  \widehat{u}\right) \nonumber\\
		&  =\mathbf{E}\left\langle \varphi_{x}\left(  \widehat{x}\left(
		t_{N+1}\right)  ,\mathbf{E}\widehat{x}\left(  t_{N+1}\right)  \right)
		,\xi\left(  t_{N+1}\right)  \right\rangle +\mathbf{E}\left\langle \varphi
		_{y}\left(  \widehat{x}\left(  t_{N+1}\right)  ,\mathbf{E}%
		\widehat{x}\left(  t_{N+1}\right)  \right)  ,\mathbf{E}\xi\left(
		t_{N+1}\right)  \right\rangle\nonumber \\
		&  +\mathbf{E}\sum_{t=t_{0}}^{t_{N}}\Big(  \left\langle \widehat{l}_{x}\left[  t\right]  ,\xi\left(  t\right)
		\right\rangle +\textbf{E}\left\langle \widehat{l}_{y}\left[  t\right]  ,\mathbf{E}\xi\left(
		t\right)  \right\rangle +\left\langle \delta_{t\theta}\widehat{l}_{u}\left[  t\right]
		,\varepsilon\Delta v\right\rangle \Big)  +o\left(  \varepsilon\right),
	\end{align}	
	where $u^{\varepsilon}(\cdot)$ is the so-called a spike (or needle) variation of $\widehat{u}(\cdot)$, defined as follows:
	\begin{align*}
		u^{\varepsilon}(t)\coloneqq
		\begin{cases*}
			\widehat{u}(t), \quad\text{for}\quad t\neq \theta,\\
			v,\qquad \text{for} \quad t=\theta.
		\end{cases*}
	\end{align*}	
\end{lemma}
\begin{proof}
	The proof is trivial and based on Lemmas \ref{lem:1} and \ref{lem:2}. By the first-order
	Taylor expansion, we get
	\begin{align*}
		&J\left(  u^{\varepsilon}\right)  -J\left(  \widehat{u}\right) \\
		  &=\mathbf{E}\left(  \varphi\left(  x^{\varepsilon}\left(  t_{N+1}\right)
		,\mathbf{E}x^{\varepsilon}\left(  t_{N+1}\right)  \right)  -\varphi\left(
		\widehat{x}\left(  t_{N+1}\right)  ,\mathbf{E}\widehat{x}\left(
		t_{N+1}\right)  \right)  \right) \\ &+\mathbf{E}\sum_{t=t_{0}}^{t_{N}}
		\left(  l^{\varepsilon}\left[  t\right]  -\widehat{l}\left[  t\right]  \right)
		\\
		&  =\mathbf{E}\left\langle \varphi_{x}\left(  \widehat{x}\left(
		t_{N+1}\right)  ,\mathbf{E}\widehat{x}\left(  t_{N+1}\right)  \right)  ,x^{\varepsilon}\left(
		t_{N+1}\right)  -\widehat{x}\left(  t_{N+1}\right)  \right\rangle\\ &+\left\langle\mathbf{E}\varphi_{y}\left(  \widehat{x}\left(  t_{N+1}\right)
		,\mathbf{E}\widehat{x}\left(  t_{N+1}\right)  \right)  ,\textbf{E}x^{\varepsilon}\left(
		t_{N+1}\right)  -\textbf{E}\widehat{x}\left(  t_{N+1}\right)  \right\rangle \\
		&  +\mathbf{E}\int_{0}^{1}\left\langle \varphi_{x}\left(\widehat{x}(t_{N+1})+\lambda(x^{\varepsilon}(t_{N+1})-\widehat{x}(t_{N+1}))  ,\mathbf{E}\widehat{x}\left(
		t_{N+1}\right)  \right) ,x^{\varepsilon}\left(  t_{N+1}\right)  -\widehat{x}\left(
		t_{N+1}\right)  \right\rangle d\lambda\\
		&+\mathbf{E}\int_{0}^{1}\left\langle \textbf{E}\varphi_{y}(\widehat{x}(t_{N+1}),\textbf{E}\widehat{x}(t_{N+1})+\lambda(\textbf{E}x^{\varepsilon}(t_{N+1})-\textbf{E}\widehat{x}(t_{N+1})) ) , \textbf{E}x^{\varepsilon}\left(  t_{N+1}\right)  -\textbf{E}\widehat{x}\left(
		t_{N+1}\right)  \right\rangle d\lambda\\
		&  +\mathbf{E}\sum_{t=t_{0}}^{t_{N}}\Big(\left\langle \widehat{l}_{x}\left[  t\right] ,x^{\varepsilon}\left(  t\right)  -\widehat{x}\left(
		t\right)\right\rangle+\mathbf{E}\left\langle \widehat{l}_{y}\left[  t\right],\textbf{E}x^{\varepsilon}\left(  t\right)  -\textbf{E}\widehat{x}\left(
		t\right)  \right\rangle\Big)
		+\mathbf{E}\sum_{t=t_{0}}^{t_{N}}
		\left\langle \widehat{l}_{u}\left[  t\right]   ,\delta_{t\theta}\varepsilon \Delta v\right\rangle\\
		& +\mathbf{E}\sum_{t=t_{0}}^{t_{N}}
		\left\langle \widehat{l}_{x}\left[  t;\varepsilon\right]  +\mathbf{E}
		\widehat{l}_{y}\left[  t;\varepsilon\right]  ,x^{\varepsilon}\left(  t\right)
		-\widehat{x}\left(  t\right)  \right\rangle +\mathbf{E}\sum_{t=t_{0}}^{t_{N}}
		\left\langle\widehat{l}_{u}\left[  t;\varepsilon\right],\delta_{t\theta}\varepsilon \Delta v\right\rangle.
	\end{align*}
	Therefore, we get the desired result \eqref{deltacost} with $o(\varepsilon)$ which is defined as follows:
	\begin{align*}
		o(\varepsilon)&\coloneqq  \mathbf{E}\int_{0}^{1}\left\langle \varphi_{x}\left(\widehat{x}(t_{N+1})+\lambda(x^{\varepsilon}(t_{N+1})-\widehat{x}(t_{N+1}))  ,\mathbf{E}\widehat{x}\left(
		t_{N+1}\right)  \right) ,x^{\varepsilon}\left(  t_{N+1}\right)  -\widehat{x}\left(
		t_{N+1}\right)  \right\rangle d\lambda\\
		&+\mathbf{E}\int_{0}^{1}\left\langle \textbf{E}\varphi_{y}(\widehat{x}(t_{N+1}),\textbf{E}\widehat{x}(t_{N+1})+\lambda(\textbf{E}x^{\varepsilon}(t_{N+1})-\textbf{E}\widehat{x}(t_{N+1})) ) , \textbf{E}x^{\varepsilon}\left(  t_{N+1}\right)  -\textbf{E}\widehat{x}\left(
		t_{N+1}\right)  \right\rangle d\lambda\\
		& +\mathbf{E}\sum_{t=t_{0}}^{t_{N}}
		\left\langle \widehat{l}_{x}\left[  t;\varepsilon\right]  +\mathbf{E}
		\widehat{l}_{y}\left[  t;\varepsilon\right]  ,x^{\varepsilon}\left(  t\right)
		-\widehat{x}\left(  t\right)  \right\rangle +\mathbf{E}\sum_{t=t_{0}}^{t_{N}}
		\left\langle\widehat{l}_{u}\left[  t;\varepsilon\right],\delta_{t\theta}\varepsilon \Delta v\right\rangle.
	\end{align*}
	Hence, our conclusion follows.
\end{proof}
Now we are in a position to complete the proof of first-order necessary conditions as stated in Theorem \ref{thm:1st order}. As in the deterministic case the proof is based on the following identity:
\begin{equation}\label{identity}
	\mathbf{E}\left\langle p\left(  t_{N+1}\right)  ,\xi\left(  t_{N+1}\right)
	\right\rangle =\sum_{t=t_{0}}^{t_{N}}\mathbf{E}\left\langle \Delta p\left(
	t\right)  ,\xi\left(  t\right)  \right\rangle +\sum_{t=t_{0}}^{t_{N}
	}\mathbf{E}\left\langle p\left(  t+h\right)  , \Delta\xi\left(  t\right)
	\right\rangle .	
\end{equation}
Observe from \eqref{bse} and \eqref{xi}, respectively that 

\begin{align*}
	\Delta p(t)&=-h\Big(\mathbf{E}\left\{  \widehat{f}_{x}^{\intercal}\left[  t\right]
	p\left(  t+h\right)  \mid\mathfrak{F}_{t}\right\} +\mathbf{E}\left\{
	\widehat{f}_{y}^{\intercal}\left[  t\right]  p\left(  t+h\right)  \right\}\Big)\\
	&+\textbf{E}\widehat{l}_{x}[t]+\textbf{E}\left\{\widehat{l}_{y}[t]\right\}-\mathbf{E}\sum_{j=1}^{d}\Big(\left(  \widehat{\sigma}_{x}^{j}\left[  t\right] \right)  ^{\intercal
	}q^{j}\left(  t\right)  +\mathbf{E}\left\{  \left(  \widehat{\sigma}_{y}^{j}\left[
	t\right]  \right)  ^{\intercal}q^{j}\left(  t\right)  \right\}\Big) 
\end{align*}
and 
\begin{align*}
	\Delta \xi(t)&=h \Big(\widehat{f}_{x}\left[  t\right]
	\xi\left(  t\right)  +\widehat{f}_{y}\left[  t\right]  \mathbf{E}\xi\left(  t\right)
	+\delta_{t\theta}\widehat{f}_{u}\left[  t\right]  \varepsilon\Delta v\Big)\\
	&+\sum_{j=1}^{d}	\mathbf{E} \Big( \widehat{\sigma}_{x}^{j}\left[  t\right]  \xi\left(  t\right)  +\widehat{\sigma}_{y}
	^{j}\left[  t\right]  \mathbf{E}\xi\left(  t\right)  +\delta_{t\theta}
	\widehat{\sigma}_{u}^{j}\left[  t\right]  \varepsilon v\Big)(w_{h}^{j}(t))^{\intercal}  
\end{align*}
Then we have that
\begin{align}\label{rhs1}
	\mathbf{E}\left\langle \Delta p\left(
	t\right)  ,\xi\left(  t\right)  \right\rangle&=-h\textbf{E}\left\langle \mathbf{E}\left\{  \widehat{f}_{x}^{\intercal}\left[  t\right]
	p\left(  t+h\right)  \mid\mathfrak{F}_{t}\right\} +\mathbf{E}\left\{
	\widehat{f}_{y}^{\intercal}\left[  t\right]  p\left(  t+h\right)  \right\}  ,\xi\left(
	t\right)  \right\rangle\nonumber\\
	& +\textbf{E}\langle \widehat{l}_{x}[t]+\textbf{E}\left\lbrace \widehat{l}_{y}[t]\right\rbrace ,\xi(t)\rangle\\
	& -\mathbf{E}\sum_{j=1}^{d}
	\left\langle \left(  \widehat{\sigma}_{x}^{j}\left[  t\right]  \right)  ^{\intercal
	}q^{j}\left(  t\right)  +\mathbf{E}\left\{  \left(  \widehat{\sigma}_{y}^{j}\left[
	t\right]  \right)  ^{\intercal}q^{j}\left(  t\right)  \right\}  ,\xi\left(
	t\right)  \right\rangle. \nonumber
\end{align}
Also, we obtain that 
\begin{align}\label{rhs2}
	\mathbf{E}\left\langle p\left(  t+h\right)  , \Delta\xi\left(  t\right)
	\right\rangle&=h\mathbf{E}\left\langle p\left(  t+h\right)  ,\widehat{f}_{x}\left[  t\right]
	\xi\left(  t\right)  +\widehat{f}_{y}\left[  t\right]  \mathbf{E}\xi\left(  t\right)
	+\delta_{t\theta}\widehat{f}_{u}\left[  t\right]  \varepsilon\Delta v\right\rangle \\
	&+
	\sum_{j=1}^{d}	\mathbf{E}\left\langle p\left(  t+h\right)  w_{h}^{j}\left(  t\right)
	,\left(  \widehat{\sigma}_{x}^{j}\left[  t\right]  \xi\left(  t\right)  +\widehat{\sigma}_{y}
	^{j}\left[  t\right]  \mathbf{E}\xi\left(  t\right)  +\delta_{t\theta}
	\widehat{\sigma}_{u}^{j}\left[  t\right]  \varepsilon v\right)  \right\rangle.\nonumber
\end{align}
Substituting \eqref{rhs1} and \eqref{rhs2} into \eqref{identity} yields that

\begin{align*}
	&\mathbf{E}\left\langle p\left(  t_{N+1}\right)  ,\xi\left(  t_{N+1}\right)
	\right\rangle\\	&=\sum_{t=t_{0}}^{t_{N}}\textbf{E}\langle \widehat{l}_{x}[t]+\textbf{E}\left\lbrace \widehat{l}_{y}[t]\right\rbrace ,\xi(t)\rangle\\
	&+\sum_{t=t_{0}}^{t_{N}}\Big(h\mathbf{E} \left\langle
	\delta_{t\theta}p\left(  t+h\right)  ,\widehat{f}_{u}\left[  t\right]  \varepsilon\Delta v\right\rangle
	+\sum_{j=1}^{d}
	\mathbf{E}\left\langle \delta_{t\theta} p\left(  t+h\right)  w_{h}^{j}\left(  t\right)
	,\widehat{\sigma}_{u}^{j}\left[  t\right]  \varepsilon\Delta v\right\rangle\Big).
\end{align*}
So, we obtain for all positive integer $N$ that that
\begin{align*}
	\mathbf{E}\left\langle p\left(  t_{N+1}\right)  ,\xi\left(  t_{N+1}\right)
	\right\rangle&=\sum_{t=t_{0}}^{t_{N}}\textbf{E}\langle \widehat{l}_{x}[t]+\textbf{E}\left\lbrace \widehat{l}_{y}[t]\right\rbrace ,\xi(t)\rangle\\
	&+\sum_{t=t_{0}}^{t_{N}}\varepsilon\mathbf{E}\left\langle h\widehat{f}_{u}^{\intercal}\left[  \theta\right]
	\mathbf{E}\left\{  p\left(  \theta+h\right)  \mid\mathfrak{F}_{\theta
	}\right\}  +\sum_{j=1}^{d}\left(  \widehat{\sigma}_{u}^{j}\left[  \theta\right]  \right)  ^{\intercal
	}q^{j}\left(  \theta\right)  ,\Delta v\right\rangle.
\end{align*}
Therefore, by taking into account of the expression of $p(t_{N+1})$ in \eqref{bse}, we have
\begin{align*}
	&  \mathbf{E}\left\langle \varphi_{x}\left(  \widehat{x}\left(  t_{N+1}\right)  ,\mathbf{E}\widehat{x}\left(  t_{N+1}\right)  \right)  ,\xi\left(
	t_{N+1}\right)  \right\rangle +\mathbf{E}\left\langle \varphi_{y}\left(
	\widehat{x}\left(  t_{N+1}\right)  ,\mathbf{E}\widehat{x}\left(
	t_{N+1}\right)  \right)  ,\mathbf{E}\xi\left(  t_{N+1}\right)  \right\rangle
	\\
	&  =-\sum_{t=t_{0}}^{t_{N}}\Big(\textbf{E}\langle \widehat{l}_{x}[t]+\textbf{E}\left\lbrace \widehat{l}_{y}[t]\right\rbrace ,\xi(t)\rangle-\varepsilon \mathbf{E}\left\langle h\widehat{f}_{u}^{\intercal}\left[
	\theta\right]  p\left(  \theta+h\right)  +\left( \widehat{\sigma}_{u}^{j}\left[
	\theta\right]  \right)  ^{\intercal}q^{j}\left(  \theta\right),\Delta v\right\rangle\Big).
\end{align*}
Thus, according to Lemma \ref{lem:3}, it follows from the fact that 
\begin{equation}\label{vineq}
	0\leq J\left(  u^{\varepsilon}\right)  -J\left(  \widehat{u}\right)  =-\varepsilon
	\mathbf{E}\left\langle H_{u}\left(  \theta,\widehat{u}\left(  \theta\right)
	\right)  ,\Delta v\right\rangle +o\left(  \varepsilon\right),\quad \theta \in \mathbb{T},
\end{equation}
where 
\begin{align*}
	H_{u}(\theta,\widehat{u}(\theta))=\varepsilon \left\langle h\widehat{f}_{u}^{\intercal}\left[
	\theta\right]  p\left(  \theta+h\right)  +\left(  \widehat{\sigma}_{u}^{j}\left[
	\theta\right]  \right)  ^{\intercal}q^{j}\left(  \theta\right),\Delta v\right\rangle-\varepsilon\left\langle \widehat{l}_{u}[\theta],\Delta v \right\rangle.
\end{align*}
Finally, we deduce the variational inequality \eqref{max1} from \eqref{vineq}. The proof is complete.

	\section{Sufficient conditions for optimality} \label{sec:5}

In this section, we will show that the general controlled discrete-time stochastic systems formulated earlier, the maximum condition in terms of $H$ function plus some convexity conditions constitute sufficient conditions for optimality.

We first introduce additional assumptions:

\begin{description}
	\item[(A5.1)] The function $\varphi$ is convex in $(x,y)$ and functions $f,\sigma,l$ are convex;
	\item[(A5.2)] The Hamiltonian $H$ is convex in $(x,y,v)$;
	\item[(A5.3)] The functions $f_{y},\sigma_{y},\varphi_{y}, l_{y}$ are non-negative.
\end{description}
\begin{theorem}
	Assume the conditions (A5.1)–(A5.3) are satisfied and let $\widehat{u}\in \mathfrak{U}_{ad}$ with state trajectory $\widehat{x}$ be given and such that there exist solutions $(p(\cdot),q^{j}(\cdot)), j=1,\ldots,d$ to the adjoint equation. Then, if 
	\begin{equation}\label{sufH}
		H(t,\widehat{x},\widehat{u}, p,q)=\inf_{v\in \mathcal{U}(t)}H(t,\widehat{x},v,p,q),
	\end{equation}
	then for all $t\in \mathbb{T}$, $\mathbb{P}$-a.s., $\widehat{u}$ is an optimal control.
\end{theorem}
\begin{remark}
	By assumption $(A 5.2)$, the conditions \eqref{max1} and \eqref{sufH} are equivalent.
\end{remark}
\begin{proof}
	We use the same short-hand notations which are defined in Section 4. Moreover, we denote 
	\begin{align*}
		H(t)&=H(t,x(t),u(t),p(t),q(t)),\\
		\widehat{H}(t)&=H(t,\widehat{x}(t),\widehat{u}(t),p(t),q(t)).
	\end{align*}
	Since $\varphi$ is convex, the following inequality holds:
	\allowdisplaybreaks
	\begin{align}\label{deltaJ}
		&  J\left(  \widehat{u}\right)  -J\left(  u\right) \\
		&  \leq-\mathbf{E}\left\langle \varphi_{x}\left(  \widehat{x}\left(
		t_{N+1}\right)  ,\mathbf{E}\widehat{x}\left(  t_{N+1}\right)  \right)
		-\mathbf{E}\varphi_{y}\left(  \widehat{x}\left(  t_{N+1}\right)
		,\mathbf{E}\widehat{x}\left(  t_{N+1}\right)  \right)  ,\widehat{x}\left(
		t_{N+1}\right)  -x\left(  t_{N+1}\right)  \right\rangle \nonumber\\
		&  +\mathbf{E}\sum_{t=t_{0}}^{t_{N}}
		\left(  \widehat{l}\left[  t\right]  -l\left[  t\right]  \right).\nonumber
	\end{align}
	Therefore, the first term of \eqref{deltaJ} will be as follows:
	\begin{align*}
		&  \mathbf{E}\left\langle p\left(  t_{N+1}\right)  ,\widehat{x}\left(
		t_{N+1}\right)  -x\left(  t_{N+1}\right)  \right\rangle   =\sum_{t=t_{0}}^{t_{N}}\Delta\mathbf{E}\left\langle p\left(  t\right)
		,\widehat{x}\left(  t\right)  -x\left(  t\right)  \right\rangle\\
		& =\sum
		_{t=t_{0}}^{t_{N}}\left(  \mathbf{E}\left\langle \Delta p\left(  t\right)
		,\widehat{x}\left(  t\right)  -x\left(  t\right)  \right\rangle +\mathbf{E}\left\langle p\left(  t+h\right)  ,\Delta\widehat{x}\left(  t\right)  -\Delta
		x\left(  t\right)  \right\rangle \right) \\
		&  =-\sum_{t=t_{0}}^{t_{N}}\Big\langle \left(  h\widehat{f}_{x}^{\intercal}\left[
		t\right]  \right)  \mathbf{E}\left\{  p\left(  t+h\right)  \mid\mathfrak{F}_{t}\right\}  +\mathbf{E}\left\{  \widehat{f}_{y}^{\intercal}\left[  t\right]  p\left(
		t+h\right)  \right\} \\
		& +\sum_{j=1}^{d}\left(  \widehat{\sigma}_{x}^{j}\left[  t\right]  \right)  ^{\intercal}q^{j}\left(
		t\right)  +\mathbf{E}\left\{  \left(  \widehat{\sigma}_{y}^{j}\left[  t\right]  \right)
		^{\intercal}q^{j}\left(  t\right)  \right\}  ,\widehat{x}\left(  t\right)
		-x\left(  t\right)  \Big\rangle \\
		&+\sum_{t=t_{0}}^{t_{N}}\left\langle \left( \widehat{l}_{x}[t]+\textbf{E}\widehat{l}_{y}[t]\right),\widehat{x}\left(  t\right)
		-x\left(  t\right)  \right\rangle \\
		&  +\sum_{t=t_{0}}^{t_{N}}\mathbf{E}\left\langle p\left(  t+h\right)
		,h\left(  \widehat{f}\left[  t\right]  -f\left[  t\right]  \right)  +\sum_{j=1}^{d}\left(  \widehat{\sigma}^{j}\left[  t\right]  -\sigma^{j}\left[  t\right]
		\right)  w_{h}^{j}\left(  t\right)  \right\rangle \\
		&  =-\sum_{t=t_{0}}^{t_{N}}\Big\langle \left(  h\widehat{f}_{x}^{\intercal}\left[
		t\right]  \right)  \mathbf{E}\left\{  p\left(  t+h\right)  \mid\mathfrak{F}_{t}\right\}  +\mathbf{E}\left\{  \widehat{f}_{y}^{\intercal}\left[  t\right]  p\left(
		t+h\right)  \right\}\\
		&  +\sum_{j=1}^{d}\left(  \widehat{\sigma}_{x}^{j}\left[  t\right]  \right)  ^{\intercal}q^{j}\left(
		t\right) +\mathbf{E}\left\{  \left(  \widehat{\sigma}_{y}^{j}\left[  t\right]  \right)
		^{\intercal}q^{j}\left(  t\right)  \right\}  ,\widehat{x}\left(  t\right)
		-x\left(  t\right)  \Big\rangle \\
		&+\sum_{t=t_{0}}^{t_{N}}\left\langle \left( \widehat{l}_{x}[t]+\textbf{E}\widehat{l}_{y}[t]\right),\widehat{x}\left(  t\right)
		-x\left(  t\right)  \right\rangle  +\sum_{t=t_{0}}^{t_{N}}\mathbf{E}\left(  \widehat{H}\left[  t\right]
		-H\left[  t\right]  \right)  +\sum_{t=t_{0}}^{t_{N}}\mathbf{E}\left(
		\widehat{l}\left[  t\right]  -l\left[  t\right]  \right),
	\end{align*}
	where in the last step we have used the definition of Hamiltonian function. Finally, we differentiate the Hamiltonian and use the convexity of functions to get for all $t \in \mathbb{T}$, $\mathbb{P}-a.s.$,
	\begin{align*}
		\sum_{t=t_{0}}^{t_{N}}\mathbf{E}\left(  \widehat{H}\left[  t\right]  -H\left[
		t\right]  \right)  &\leq\sum_{t=t_{0}}^{t_{N}}\mathbf{E}\left\langle
		\widehat{H}_{x}\left[  t\right]  +\widehat{H}_{y}\left[  t\right]
		,\widehat{x}\left(  t\right)  -x\left(  t\right)  \right\rangle \\
		&+\sum
		_{t=t_{0}}^{t_{N}}\mathbf{E}\left\langle \widehat{H}_{u}\left[  t\right]
		,\left(  \widehat{u}\left(  t\right)  -u\left(  t\right)  \right)
		\right\rangle	\\
		&\leq \sum_{t=t_{0}}^{t_{N}}\mathbf{E}\left\langle
		\widehat{H}_{x}\left[  t\right]  +\widehat{H}_{y}\left[  t\right]
		,\widehat{x}\left(  t\right)  -x\left(  t\right)  \right\rangle, 
	\end{align*}
	where we have applied $\left\langle \widehat{H}_{u}\left[  t\right]
	,\left(  \widehat{u}\left(  t\right)  -u\left(  t\right)  \right)
	\right\rangle\leq 0$ in the last step due to the minimum condition \eqref{sufH}. Combining the above inequalities gives us
	\begin{align*}
		J\left(  \widehat{u}\right)  -J\left(  u\right)&  \leq \mathbf{E}\left\langle p\left(  t_{N+1}\right)  ,\widehat{x}\left(
		t_{N+1}\right)  -x\left(  t_{N+1}\right)  \right\rangle +\mathbf{E}\sum_{t=t_{0}}^{t_{N}}
		\left(  \widehat{l}\left[  t\right]  -l\left[  t\right]  \right)\\
		&\leq \sum_{t=t_{0}}^{t_{N}}\mathbf{E}\left(  \widehat{H}\left[  t\right]  -H\left[
		t\right]  \right)-\sum_{t=t_{0}}^{t_{N}}\mathbf{E}\left\langle
		\widehat{H}_{x}\left[  t\right]  +\widehat{H}_{y}\left[  t\right]
		,\widehat{x}\left(  t\right)  -x\left(  t\right)  \right\rangle \leq 0
	\end{align*}
	and thus, $\widehat{u}$ is an optimal.
\end{proof}

		\section{Application}\label{sec:7}
	Portfolio selection is to seek a best allocation of wealth among a basket of securities. This model is the foundation of modern finance theory and inspired literally hundreds of extensions and applications.
	
	Our application is concerned with a discrete-time portfolio selection model that is formulated as  production and consumption choice optimization problems. The
	objective is to maximize the expected terminal return and minimize the variance
	of the terminal wealth. By putting weights on the two criteria one obtains a single
	objective stochastic control problem which is however not in the standard form
	due to the variance term involved.

	We suppose that an investor is able to invest his wealth to produce some production, and he can get profit from the production. Denote by $x(t)$ the capital of this investor at time $t$ and by $v(t)$ the rate of consumption.
	
	Now, we consider some risk in the investment
	process
	\begin{align}\label{ex1}
		\begin{cases}
			x\left(  t+h\right)  =x\left(  t\right)  +h\left(  f\left(  x\left(  t\right)
			\right)  -\delta x\left(  t\right)  \right)  -v\left(  t\right)
			+\sigma\left(  x\left(  t\right)  \right)  w_{h}\left(  t\right)  ,\\
			x\left(  0\right)  =x_{0}\in \mathbb{R}^{+},\qquad t=t_{0},...,t_{N},
		\end{cases}
	\end{align}
	where $f\left(  x\right)  $ is the income production, $\delta$ is the
	depreciation rate of the capital, $\sigma\left(  x\right)  $ denotes the
	effect of influenced by the exogenous environment and $w_{h}\left(  t\right)
	$ are the $1$-dimentional white noises. Our objective is to choose the optimal
	consumption rate $v\left(  t\right)  \geq0$ to maximize the following
	functional $J\left(  v\right)  :$
	\begin{equation}\label{J}
		J\left(  v\right)  =\textbf{E}x\left(  t_{N+1}\right)  +\textbf{E}\sum\limits_{k=0}^{N}l\left(  v\left(  t_{k}\right)  \right)  ,	
	\end{equation}
	
	where $x\left(  t_{N+1}\right)  $ is the capital left over after consumption
	in the last period $t_{N+1}$. $l$ is the utility function given by%
	\[
	l\left(  v\right)  =\frac{\delta}{\delta-1}v^{1-\frac{1}{\delta}%
	},\ \ \ 0<\delta<1.
	\]
	Then Hamiltonian function is%
	\[
	H\left(  t,v\right)  :=l\left(  v\right)  +hp\left(  t+h\right)  \left(
	f\left(  x\left(  t\right)  \right)  -\delta x\left(  t\right)  -v\right)
	+q\left(  t\right)  \sigma\left(  x\left(  t\right)  \right)
	\]
	and the corresponding backward stochastic difference equation is%
	\[
	\left\{
	\begin{array}
		[c]{l}%
		p\left(  t\right)  =\left(  1+hf_{x}\left(  x\left(  t\right)  \right)
		\right)  \mathbf{E}\left\{  p\left(  t+h\right)  \mid\mathfrak{F}_{t}\right\}
		+\sigma_{x}\left(  x\left(  t\right)  \right)  q\left(  t\right) \\
		q\left(  t\right)  =\mathbf{E}\left\{  p\left(  t+h\right)  w_{h}\left(
		t\right)  \mid\mathfrak{F}_{t}\right\} \\
		p\left(  t_{N+1}\right)  =1,\ q\left(  t_{N}\right)  =0.
	\end{array}
	\right.
	\]
	Assume that $f\left(  x\right)  =x$ and $\sigma\left(  x\right)  =\frac{1}%
	{2}x.$ Then the Hamiltonian function is
	\begin{align*}
		H\left(  t,v\right)   &  =\frac{\delta}{\delta-1}v^{1-\frac{1}{\delta}%
		}+hp\left(  t+h\right)  \left(  x\left(  t\right)  -\delta x\left(  t\right)
		-v\right)  +\frac{1}{2}q\left(  t\right)  x\left(  t\right) .
	\end{align*}
	Solving the equation with respect to $v$:
	\begin{align*}
		H_{v}\left(  t,v\right) =v^{-\frac{1}{\delta}}-hp\left(  t+h\right)
		=0,  
	\end{align*}
	we obtain 
	\begin{equation*}
		v\left(  t\right) =h^{-\delta}p^{-\delta}\left(  t+h\right), \quad t \in \mathbb{T}.
	\end{equation*}
	By Theorem \ref{thm:1st order}, we know that $\left\lbrace v(t)\right\rbrace $ is the optimal consumption rate for the optimization problem \eqref{ex1}-\eqref{J}. 
	
	If $N=5$ then we have
	\begin{align*}
		p\left(  6h\right)   &  =1,\ q\left(  5h\right)  =0,\\
		p\left(  5h\right)   &  =h\left(  2-\delta\right)  \mathbf{E}\left\{  p\left(
		6h\right)  \mid\mathfrak{F}_{t}\right\}  +\frac{1}{2}0=h\left(  2-\delta
		\right)  ,\\
		q\left(  4h\right)   &  =0,\\
		p\left(  4h\right)   &  =h\left(  2-\delta\right)  \mathbf{E}\left\{  h\left(
		2-\delta\right)  \mid\mathfrak{F}_{t}\right\}  =h^{2}\left(  2-\delta\right)
		^{2}.
	\end{align*}
	The trajectory of $v(t)$ with $\delta=0.5$ and $h=0.5$ will be illustrated by the following figure:
	
	\begin{figure}[H]
		\centering
		\includegraphics[width=0.6\linewidth]{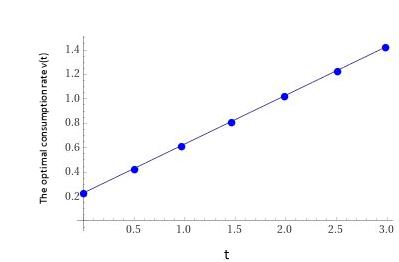}
		\caption{Trajectories of optimal consumption rate $v(t)$}
	\end{figure}
	\section{Conclusions}
	The core of this paper is to study a new version of maximum principle of discrete-time mean-field stochastic systems and to provide a proof of main result stated in Theorem \ref{thm:1st order}. As a consequence of the main theorem constructing stochastic maximum principle to the discrete-time stochastic systems, by comparing these results with some existing results in the literature proved from different point of view.
	
	The main contributions of our work are described in detail as follows:
	\begin{itemize}
		\item We established a new version of the maximum principle for discrete-time mean-field stochastic optimal control problems and the first-order necessary and sufficient optimality conditions for discrete-time stochastic optimal control problems; 
		\item We introduced discrete-time backward (matrix) stochastic equation. Based on the discrete-time backward stochastic equations, we have obtained necessary first-order and second-order optimality conditions for the stochastic discrete optimal control problem \eqref{op1}-\eqref{cost};
		\item Finally, we considered an application as a kind of optimization problem for production and consumption.
	\end{itemize} 	
	Other related research directions in maximum principle may include the various relevant topics that may be useful in the optimal control theory of stochastic differential equations with mean-field type, e.g., one can consider a maximum principle of a fractional analogue of discrete and continuous time stochastic differential equations. Our method can also be applied to more complicated discrete-time stochastic optimal control problems, for example, problems with delays, terminal constraint problems, and problems with neutral term.
	
Although there are many articles on the maximum principle of stochastic and deterministic systems, there still remain many other interesting open problems concerning their fractional analogues, which can be extended by methods analogous to those used for fractional derivations of Caputo and Riemann-Liouville type. To this end, one can consider the method given in \cite{yusubov-mahmudov} to study the optimal control problem in which a dynamical system is controlled by a nonlinear Caputo fractional state equation.

\end{document}